\documentclass{amsart}
\usepackage{amsfonts}
\usepackage{amssymb}
\usepackage{times}
\usepackage{url}
\usepackage{float}
\usepackage{longtable}
\usepackage{tasks}
\usepackage{xcolor}
\usepackage{tkz-graph}
\usepackage{hyperref}
\usepackage{tikz,pgf}
\usepackage{graphicx}
\usetikzlibrary{shapes}
\usetikzlibrary{calc,arrows.meta,positioning,arrows}

\def\B{\mathcal{B}}
\def\C{\mathcal{C}}

\def\IF{\ensuremath{\mathbb F}}

\usepackage{cite}
\usepackage{graphicx}

\makeatletter
\newcommand{\bigperp}{%
  \mathop{\mathpalette\bigp@rp\relax}%
  \displaylimits
}

\newcommand{\bigp@rp}[2]{%
  \vcenter{
    \m@th\hbox{\scalebox{\ifx#1\displaystyle2.1\else1.5\fi}{$#1\perp$}}
  }%
}
\makeatother

\NewTasksEnvironment[label =\arabic*, label-width=0.9em,label-format=\bfseries]{tabbedEnum}[\item](2)

\usepackage{amsmath}
\usepackage{eurosym}
\usepackage{geometry}

\usepackage{caption,booktabs}

\newtheorem{theorem}{Theorem}[section]

\newtheorem{corollary}{Corollary}[section]

\newtheorem{definition}{Definition}[section]
\newtheorem{example}{Example}[section]

\newtheorem{lemma}{Lemma}[section]

\newtheorem{remark}{Remark}[section]

\begin{document}

\title{Decomposition of linear operators on pre-euclidean spaces by means of graphs}

\thanks{The second author was supported by the Centre for Mathematics of the University of Coimbra - UIDB/00324/2020, funded by the Portuguese Government through FCT/MCTES. Third and fourth author is supported by the PCI of the UCA `Teor\'\i a de Lie y Teor\'\i a de Espacios de Banach' and by the PAI with project number FQM298. Third author also is supported by the project FEDER-UCA18-107643 and by the Spanish project `Algebras no conmutativas y de caminos de Leavitt. Algebras de evolución. Estructuras de Lie y variedades
de Einstein'.}

\author[H. Abdelwahab]{Hani Abdelwahab}
\address{Hani Abdelwahab.
\newline \indent Mansoura University, Faculty of Science, Department of Mathematics (Egypt).}
\email{{\tt haniamar1985@gmail.com}}

\author[E. Barreiro]{Elisabete Barreiro}
\address{Elisabete~Barreiro. \newline \indent University of Coimbra, CMUC, Department of Mathematics,  FCTUC,
Largo D. Dinis
3000-143 Coimbra (Portugal).}
\email{{\tt mefb@mat.uc.pt}}

\author[A.J. Calder\'on]{Antonio J. Calder\'on}
\address{Antonio J. Calder\'on. \newline \indent University of Cádiz, Department of Mathematics, Puerto Real (Spain).}
\email{{\tt ajesus.calderon@uca.es}}

\author[J.M. Sánchez]{José María Sánchez}
\address{José María Sánchez. \newline \indent University of Cádiz, Department of Mathematics, Puerto Real (Spain).}
\email{{\tt txema.sanchez@uca.es}}

%%%%%%%%%%%%%%%%%%%%%%%%%%%%%%%%%%%%%%%%%%%%%%%%%%%%%%%%%%%%%%%%%%%%%

\thispagestyle{empty}

\begin{abstract}
In this work we study a linear operator $f$  on a pre-euclidean space $\mathcal{V}$ by using properties of a corresponding graph. Given a basis $\B$ of $\mathcal{V}$, we present a decomposition of $\mathcal{V}$ as an orthogonal direct sum of certain linear subspaces $\{U_i\}_{i \in I}$, each one admitting a basis inherited from $\B$, in such way that $f = \sum_{i \in I}f_i$, being each $f_i$ a linear operator satisfying certain conditions respect with $U_i$.   Considering new hypothesis, we assure the existence of an isomorphism between the graphs associated to $f$ relative to two different bases. We also study the minimality of $\mathcal{V}$ by using the  graph associated to $f$ relative to $\B$.

\bigskip

{\it 2020MSC}: 47A65, 47B37, 05C90.

\end{abstract}

\maketitle

%%%%%%%%%%%%%%%%%%%%%%%%%%%%%%%%%%%%%%%%%%%%%%%%%%%%%%%%%%%%%%%%%%%%%%%%%%%%%%%%%%%%%%%%%%%%%
\section{Introduction}
\label{sec1}
%%%%%%%%%%%%%%%%%%%%%%%%%%%%%%%%%%%%%%%%%%%%%%%%%%%%%%%%%%%%%%%%%%%%%%%%%%%%%%%%%%%%%%%%%%%%%

This paper is motivated by the following problem: If we consider a family $\{f_i : \mathcal{V}_i \to \mathcal{V}_i\}_{i\in I}$ of linear operators, where any $\mathcal{V}_i$ is a pre-euclidean space, then we can construct in a natural way a new linear operator $$f: \bigperp_{i \in I} \mathcal{V}_i \to 
 \bigperp_{i \in I} \mathcal{V}_i$$ as $f=\sum_{i\in I} f_i$,   where $\mathcal{V} := \bigperp_{i\in I} \mathcal{V}_i$ is the pre-euclidean space defined by  componentwise operations. But, what about the converse? That is, if $$f : \mathcal{V} \to \mathcal{V}$$ is a linear operator on a pre-euclidean space. Can we find a family $\{\mathcal{V}_i\}_{i\in I}$ of pre-euclidean spaces and a family of linear operators $\{f_i : \mathcal{V}_i \to \mathcal{V}_i\}_{i\in I}$ in such a way that 
\begin{equation}\label{de} 
\mathcal{V} = \bigperp_{i \in I}  \mathcal{V}_i \hspace{0.3cm}  {\rm and} \hspace{0.3cm}  f = \sum_{i\in I} f_i?
\end{equation}

The aim of the present work is to study this problem by giving a positive answer. A pre-euclidean space is simply a linear space provided with a bilinear form, hence our work covers a wide range of structures. 
We also note that we use, as a tool for our study, techniques of graphs. This allows us to get the decomposition \eqref{de} in an easy way, just by looking to the graph associated to $f$ (and a fixed basis). 
Observe that this result gives us the opportunity of recover a (possible) large linear operator $f : \mathcal{V} \to \mathcal{V}$, from a family of easier linear operators $f_i : \mathcal{V}_i \to \mathcal{V}_i$, in a  visual and computable way, what we hope will be useful in any area dealing with linear operators on a vector space endowed with a bilinear map. In recent years, the use of graphs has increased in order to apply them in other areas \cite{Leibniz_grafos, Yografos,apply6, apply8, apply13}. Concretely, the applications of the graphs for the study of linear operators and algebras \cite{graph_oper1, graph_oper2, graph_oper3, graph_oper4, graph_oper5, graph_oper6, graph_oper7}. We highlight that this topic is currently very active as some recently published articles show \cite{apply1, apply2, apply3, apply4, apply5, apply7, apply9, apply10, apply11, apply12}.

More in detail, for a linear operator $f : \mathcal{V} \to \mathcal{V}$  on a pre-euclidean space $\mathcal{V}$ with fixed a basis $\B$, we obtain its decomposition as an orthogonal direct sum of certain linear subspaces $\{U_i\}_{i \in I}$, each one admitting a basis inherited from $\B$, in such way that $f$ is decomposed as $f = \sum_{i \in I}f_i$, being each $f_i$ a linear operator satisfying certain conditions relative to $U_i$. Also, for the linear operator $f$, we present conditions in order to guarantee the existence of an isomorphism between the graphs associated to $f$ relative to two different bases of $\mathcal{V}$. Finally, we analyze the minimality property for $\mathcal{V}$ by using the  graph associated to $f$ relative to $\B$.

The paper is organized as follows. Section \ref{Sect1} contains basic notions needed in the sequel. In Section \ref{Sect2} we associate a graph $\Gamma(f,\B)$ to any linear operator $f$ defined on a pre-euclidean space $(\mathcal{V},\langle \cdot, \cdot \rangle)$ with a fixed basis $\B$. In addition, we introduce the notion of $f$-indecomposable in order to give a characterization using the connectivity of $\Gamma(f,\B)$. In Section \ref{Sect3} we give the definition of $f$-equivalence under which the graphs $\Gamma(f,\B)$ and $\Gamma(f,\B')$ are isomorphic, for two different bases $\B $ and $ \B'$ of $\mathcal{V}$. Also, we relate this properties with the definition of equivalent decomposition for $f$. In Section \ref{Sect4} we analyze the minimality property for $\mathcal{V}$ by using $\Gamma(f,\B)$, the graph associated  to $f$ relative to $\B$. Finally, we present an entire conclusions section where we critically highlight our contribution, and identifying strengths and weaknesses and proposing paths for future research.  

%%%%%%%%%%%%%%%%%%%%%%%%%%%%%%%%%%%%%%%%%%%%%%%%%%%%%%%%%%%%%%%%%%%%%%%%%%%%%%%%%%%%%%%%%%%%%
\section{Basic definitions}
\label{Sect1}
%%%%%%%%%%%%%%%%%%%%%%%%%%%%%%%%%%%%%%%%%%%%%%%%%%%%%%%%%%%%%%%%%%%%%%%%%%%%%%%%%%%%%%%%%%%%%
 
Throughout this paper, $\mathbb{F}$ denotes an arbitrary field and all vector spaces are assumed to be arbitrary dimensional and over base field $\mathbb{F}$.
% Given a $\mathbb{F}$-linear space $A$, a bilinear form $\langle \cdot , \cdot \rangle : A \times A \to \mathbb{F}$ is said:\begin{enumerate}\item[i.]  {\it positive definite} if $\langle x,x \rangle > 0$ for all $x \in A \setminus \{0\}$,\item[ii.]  {\it symmetric} if $\langle x, y\rangle = \langle y, x \rangle$ for any $x,y \in A$. \end{enumerate} A positive definite bilinear form $\langle \cdot , \cdot \rangle : A \times A \to \mathbb{F}$ is called a {\it scalar product}.  

\begin{definition}\rm
A \emph{pre-euclidean space} is a pair $(\mathcal{V}, \langle \cdot , \cdot \rangle)$, where $\mathcal{V}$ is a $\mathbb{F}$-vector space and $\langle \cdot , \cdot \rangle : \mathcal{V} \times \mathcal{V} \to \mathbb{F}$ is a bilinear form.
\end{definition}

\begin{example}\rm
$\mathbb{R}$ endowed with the bilinear form $\langle \cdot, \cdot \rangle : \mathbb{R} \times \mathbb{R} \to \mathbb{R}$ given as $\langle x,y \rangle := \lambda xy$, for $x,y\in \mathbb{R}$ and fixed $\lambda \in \mathbb{R}$, is a pre-euclidean space.
\end{example}

A pre-euclidean space over the field $\mathbb{R}$ endowed with an scalar product is a pre-Hilbert space. Therefore, the results in this paper apply  to pre-Hilbert spaces.  A {\it pre-euclidean subspace} of $(\mathcal{V}, \langle \cdot , \cdot \rangle)$ is a linear subspace $U$ of $\mathcal{V}$ endowed with the bilinear form $\langle \cdot ,\cdot \rangle |_{U \times U}$.
Additionally, given two pre-euclidean spaces $(\mathcal{V}, \langle \cdot, \cdot \rangle_{\mathcal V})$ and $ (\mathcal{W}, \langle \cdot, \cdot \rangle_{\mathcal W})$, a {\it morphism} from $\mathcal{V}$ to $\mathcal{W}$ is a linear map $\phi : \mathcal{V} \to \mathcal{W}$ satisfying $\langle x,y \rangle_{\mathcal V} = \langle \phi(x),\phi(y) \rangle_{\mathcal W}$ for $x,y\in \mathcal{V}$. An {\em isomorphism} is a bijective morphism from $\mathcal{V}$ to $\mathcal{W}$.  More, an {\em automorphism} is an isomorphism from $\mathcal{V}$ to itself.

\begin{definition} \rm
Let $(\mathcal{V}, \langle \cdot , \cdot \rangle)$ be a pre-euclidean space.
\begin{enumerate}
\item[i.] We say that two elements $x,y \in \mathcal{V}$ are {\it orthogonal} if $\langle x,y \rangle = 0$.
\item[ii.] The vector subspaces $U$ and $W$ of $\mathcal{V}$ are {\it orthogonal} if $\langle u,w \rangle = 0$ for $u \in U, w \in W$. In this case, we denote by $\langle U,W \rangle = \{0\}$.
\item[iii.] $\mathcal{V}$ is an {\it orthogonal direct sum} of linear subspaces $U_i$  of $\mathcal{V}$, with $i \in I$, denoted as $$\mathcal{V} = \bigperp_{i \in I}U_i,$$ if $\mathcal{V}$ decomposes as a direct sum $\mathcal{V} = \bigoplus_{i \in I}U_i$ of linear subspaces $U_i$ such that $\langle U_i,U_j \rangle = \{0\}$ whenever $i\neq j$.
\end{enumerate}
\end{definition}

\section{Linear operator on a pre-euclidean space and graphs. Decomposition Theorem}
\label{Sect2}
%%%%%%%%%%%%%%%%%%%%%%%%%%%%%%%%%%%%%%%%%%%%%%%%%%%%%%%%%%%%%%%%%%%%%%%%%%%%%%%%%%%%%%%%%%%%%

We recall  that a (directed) graph is a pair $(V,E)$ where $V$ is a set of vertices and $E\subset V\times V$ a set of (directed) edges connecting the vertices.  
\begin{definition}\label{def_graph_associated}\rm
Let $f : \mathcal{V} \to \mathcal{V}$ be a linear operator on a pre-euclidean space $(\mathcal{V}, \langle \cdot , \cdot \rangle)$ with fixed basis $\mathcal{B}=\{e_i\}_{i \in I}$. The directed graph associated to $f$ relative to $\B$ is $\Gamma(f,\B) := (V,E)$, where $V := \B$ and
$$E := \Bigl\{(e_i,e_j) \in V \times V : \{\langle e_i,e_j \rangle , \langle e_j,e_i\rangle \} \neq \{0\} \mbox{ or } f(e_i) = \sum_j \lambda_j e_j \mbox{ for some  } 0 \neq \lambda_j \in \mathbb{F} \Bigr\}.$$
We say that $\Gamma(f,\B)$ is the {\it (directed) graph associated to} $f$ relative to basis $\B$.
\end{definition}

\begin{example}  \label{example1} \rm
Let $(\mathcal{V},\langle \cdot ,\cdot \rangle )$ be the pre-euclidean space over $\mathbb{R}$ with a fixed basis $\B = \{e_1, e_2, \ldots , e_5\}$  such that $\langle e_2,e_5 \rangle =7$ and the rest zero. Let $f : \mathcal{V} \to \mathcal{V}$ be the linear operator defined as $$f(e_1) = f(e_3) = f(e_5) := e_1+2e_3+e_5, \hspace{1cm} f(e_2) = f(e_4) := -e_4.$$

%\begin{equation*} f (e_i)=\left\{ \begin{tabular}{ll} $e_1 + e_3 + e_5$ & if $i$ is odd, \\  $e_2 + e_4$ & if $i$ is even.% \end{tabular}\right. \end{equation*}
 
\noindent Then the associated graph $\Gamma(f,\B)$ is:
 
\bigskip

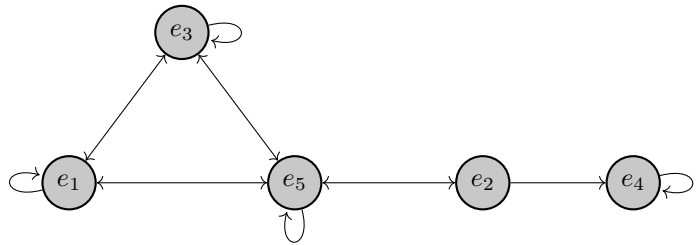
\begin{center}
\begin{tikzpicture}\label{GRAPH1}[scale=0.8, transparency group=knockout]
\begin{scope}[every node/.style={circle,thick,draw}]
\node[shape=circle,fill={rgb,255:gray,300; white,50},draw=black] (E1) at (-1.5,0) {$e_1$};
\node[shape=circle,fill={rgb,255:gray,300; white,50},draw=black] (E5) at (1.5,0) {$e_5$};
\node[shape=circle,fill={rgb,255:gray,300; white,50},draw=black] (E3) at (0,2) {$e_3$};
\node[shape=circle,fill={rgb,255:gray,300; white,50},draw=black] (E2) at (4,0) {$e_2$};
\node[shape=circle,fill={rgb,255:gray,300; white,50},draw=black] (E4) at (6,0) {$e_4$};

%ARROWS
\draw[<->] (E1) edge (E5);
\draw[<->] (E5) edge (E3);
\draw[<->] (E1) edge (E3);
\draw[->] (E2) edge (E4);
\draw[<->] (E2) edge (E5);

%LOOPS
\path (E1) edge [loop left] (E1);
\path (E3) edge [loop right] (E3);
\path (E5) edge [loop below] (E5);
\path (E4) edge [loop right] (E4);
\end{scope}
\end{tikzpicture}
\end{center}
\end{example}

\begin{example}\label{new_exam}\rm
Let $(\mathcal{V},\langle \cdot, \cdot \rangle )$ be the pre-euclidean space over $\mathbb{R}$ with basis $\B:=\{v_1,v_2,v_3,v_4,v_5\}$ and bilinear form defined as $$\langle v_4,v_3 \rangle = \langle v_1,v_5 \rangle := 1, \hspace{1cm} \langle v_4,v_2 \rangle = \langle v_3,v_5 \rangle := 3.$$ Let $f : \mathcal{V} \to \mathcal{V}$ be the linear operator given as $$f(v_1)=f(v_3)=f(v_2) := 2v_1+2v_2+v_3, \hspace{1cm} f(v_4)=v_4.$$

\noindent So the graph $\Gamma(f,\B)$ is:

\bigskip

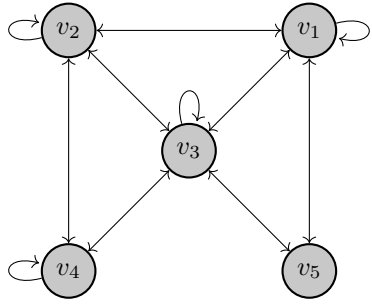
\begin{center}
\begin{tikzpicture}[scale=0.8, transparency group=knockout]
\begin{scope}[every node/.style={circle,thick,draw}]
\node[shape=circle,fill={rgb,255:gray,300; white,50},draw=black] (F) at (0,4) {$v_2$};
\node[shape=circle,fill={rgb,255:gray,300; white,50},draw=black] (H) at (2,2) {$v_3$};
\node[shape=circle,fill={rgb,255:gray,300;white,50},draw=black] (E) at (4,4) {$v_1$};
\node[shape=circle,fill={rgb,255:gray,300;white,50},draw=black] (P) at (0,0) {$v_4$};
\node[shape=circle,fill={rgb,255:gray,300;white,50},draw=black] (Q) at (4,0) {$v_5$};

%ARROWS
\draw[<->] (E) edge (Q);
\draw[<->] (E) edge (H);
\draw[<->] (E) edge (F);
\draw[<->] (H) edge (F);
\draw[<->] (H) edge (P);
\draw[<->] (H) edge (Q);
\draw[<->] (P) edge (F);

%LOOPS
\path (E) edge [loop right] (E);
\path (F) edge [loop left] (F);
\path (H) edge [loop above] (H);
\path (P) edge [loop left] (P);

\end{scope}
\end{tikzpicture}
\end{center}
\end{example}

Given two vertices $v_i,v_j \in V$, an \emph{undirected path} from $v_i$ to $v_j$ is a sequence of vertices $(v_{i_1},\dots,v_{i_n})$ with $v_{i_1}=v_i$,  $v_{i_n}=v_j$ and such that either $(v_{i_r}, v_{i_{r+1}}) \in E$ or $(v_{i_{r+1}}, v_{i_r}) \in E$, for $1\leq r \leq n-1$.
We may introduce an  equivalence relation in $V$: we say that $v_i$ is related to $v_j$ in $V$, and denote  $v_i\sim v_j$, if either $v_i= v_j$ or there exists an undirected path from $v_i$ to $v_j$. In this case, we assert that $v_i$ and $v_j$ are \emph{connected} and the equivalence class of $v_i$, denoted by $[v_i] \in V/\sim$, corresponds to a connected component $\C_{[v_i]}$ of the graph $\Gamma(f,\B)$. Therefore
\begin{equation}\label{graphdec}
\Gamma(f,\B)=\dot{\bigcup_{[v_i]\in V/\sim}} \C_{[v_i]}.
\end{equation}

\medskip

\noindent We can also  associate to any $\C_{[v_i]}$ the linear subspace
\begin{equation}
{\mathcal V}_{\C_{[v_i]}} := \bigperp_{v_j\in [v_i]}\IF v_j.\label{linsub}
\end{equation}

\begin{definition}\rm
Let $(\mathcal{V}, \langle \cdot , \cdot \rangle)$ be a pre-euclidean space with basis $\mathcal{B}$. A linear subspace $U$ of ${\mathcal V}$ admits a basis $\B'$ {\it inherited from} $\B$ if $\B'$ is a basis of $U$ satisfying $\B' \subset \B$.
\end{definition}

\begin{definition}\rm
Let $f : \mathcal{V} \to \mathcal{V}$ be a linear operator on a pre-euclidean space $(\mathcal{V}, \langle \cdot , \cdot \rangle)$ with basis $\mathcal{B}=\{e_i\}_{i \in I}$. The space ${\mathcal V}$  is {\it $f$-decomposable} with respect to   $\B$ if ${\mathcal V} = U_1 \perp U_2$, being $U_1,U_2$ non-zero linear subspaces admitting each one a basis inherited from $\B$ and $f(U_1) \subset U_1$, $f(U_2)\subset U_2$.  Otherwise, ${\mathcal V}$ is said {\it $f$-indecomposable} with respect to $\B$.
\end{definition}

\begin{example}\rm
The pre-euclidean space over $\mathbb{R}$ of Example \ref{example1} is $f$-indecomposable with respect $\B$  
\end{example}

\begin{example}\rm
Let $\mathcal{V}$ be the  pre-euclidean space defined by: the $5$-dimensional $\mathbb{C}$-vector space with basis $\B:=\{e_1,e_2,e_3,e_4,e_5\}$ and bilinear form defined as $$\langle e_1, e_3 \rangle := 4i, \hspace{2cm} \langle e_4, e_5 \rangle := 2-11i,$$ and the rest zero. We consider the linear operator $f : \mathcal{V} \to \mathcal{V}$ given as $$f(e_1):=2e_1-e_2, \hspace{2cm} f(e_2):=e_3,$$ and zero on the rest. Then by denoting $U_1,U_2$ the $\mathbb{C}$-linear subspaces of $\mathcal{V}$ with bases $\{e_1,e_2,e_3\}$, $\{e_4,e_5\}$, respectively, we easily see that $${\mathcal V} = U_1 \perp U_2 $$ and ${\mathcal V}$ is $f$-decomposable with respect to $\B$. The associated graph $\Gamma(f,\B)$ is:

\bigskip

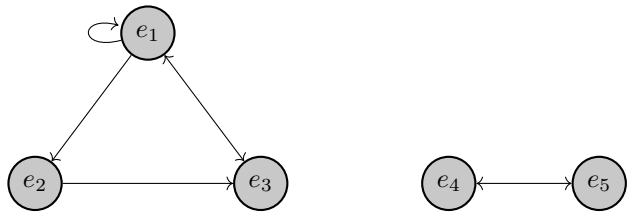
\begin{center}
\begin{tikzpicture}\label{GRAPH1B}[scale=0.8, transparency group=knockout]
\begin{scope}[every node/.style={circle,thick,draw}]
\node[shape=circle,fill={rgb,255:gray,300; white,50},draw=black] (E1) at (0,2) {$e_1$};
\node[shape=circle,fill={rgb,255:gray,300; white,50},draw=black] (E2) at (-1.5,0) {$e_2$};
\node[shape=circle,fill={rgb,255:gray,300; white,50},draw=black] (E3) at (1.5,0) {$e_3$};
\node[shape=circle,fill={rgb,255:gray,300; white,50},draw=black] (E4) at (4,0) {$e_4$};
\node[shape=circle,fill={rgb,255:gray,300; white,50},draw=black] (E5) at (6,0) {$e_5$};

%ARROWS
\draw[->] (E1) edge (E2);
\draw[<->] (E1) edge (E3);
\draw[->] (E2) edge (E3);
\draw[<->] (E4) edge (E5);

%LOOPS
\path (E1) edge [loop left] (E1);
\end{scope}
\end{tikzpicture}
\end{center}
\end{example}

\noindent A graph $(V,E)$  is \emph{connected} if  any two vertices are connected.   Equivalently, a graph is connected if and only if for every partition of its vertices into two non-empty sets, there is an edge with an endpoint in each set. 

\begin{theorem} \label{f-ind}
Let $f : \mathcal{V} \to \mathcal{V}$ be a linear operator on a pre-euclidean space $(\mathcal{V}, \langle \cdot , \cdot \rangle)$ with basis $\mathcal{B}=\{e_i\}_{i \in I}$. Then the following statements are equivalent.
\begin{enumerate}
\item[i.] The graph $\Gamma(f,\B)$ is connected.
\item[ii.] $\mathcal{V}$ is $f$-indecomposable with respect to $\B$. 
\end{enumerate}
\end{theorem}

\begin{proof}
First we suppose that the graph $\Gamma(f,\B)$ is connected. Let us assume that $\mathcal{V}$ is $f$-decomposable with respect to $\B$.
So $\mathcal{V}$ is the orthogonal direct sum $$\mathcal{V} = U_1 \perp U_2$$ of two linear subspaces $U_1$ and  $U_2$ admitting each one a basis $\B_1 := \{e_j : j \in J\}$ and $\B_2 := \{e_k : k \in K\}$, respectively, inherited from $\B$ such that $f(U_1)\subset U_1$ and $f(U_2)\subset U_2$. Hence $\B = \B_1 \; \dot{\cup} \; \B_2.$
Fix some $e_j \in \B_1$ and $e_k \in \B_2$. Since the graph $\Gamma(f,\B)$ is connected, it follows that $e_j$ is connected to $e_k$. So there exists an undirected path $$(e_j, v_{i_2}, \dots, v_{i_{n-1}}, e_k)$$ from $e_j$ to $e_k$.  From here, there are $e' = v_{i_s} \in \B_1$ and $e'' = v_{i_{s+1}} \in \B_2$ such that either $(e',e'') \in E$ or $(e'',e') \in E$. Since $\langle e',e'' \rangle = \langle e'',e'\rangle = 0$, we have either $P_{\mathbb{F}e''}\bigl(f(e') \bigr) \neq 0$ or $P_{\mathbb{F}e'}\bigl(f(e'')\bigr) \neq 0$, where $P_U : \mathcal{V} \to U$ is the projection of $\mathcal{V}$ onto the linear subspace $U$. Hence, we have either $f(U_1) \not\subset U_1$ or $f(U_2) \not\subset U_2$, in both cases it is a contradiction. Therefore, $\mathcal{V}$ is $f$-indecomposable with respect to $\B$.

Conversely, let us suppose that $\mathcal{V}$ is $f$-indecomposable   with  respect to $\B$ and $\Gamma(f,\B)$ is not connected. Then there exists a partition $\mathcal{B} = \B_1 \; \dot{
\cup} \; \B_2$ such that both $(x,y)$ and $ (y,x)$ are not in  $ E$, for any $x\in \B_1$ and $y\in \B_2$. Set $U_1:=\bigoplus_{x\in \B_1}\mathbb{F}x$ and $U_2:=\bigoplus_{y
\in \B_2}\mathbb{F}y$. Then we have $\langle x,y\rangle = \langle
y,x\rangle = 0$ for any $x \in \B_1$ and $y \in \B_2$. Moreover, $f(U_1)\subset U_1$ and $f(U_2)\subset U_2$. So $\mathcal{V}$ is the orthogonal direct sum 
\begin{equation*}
\mathcal{V} = U_1 \perp U_2,
\end{equation*}%
of two linear subspaces $U_1$ and $U_2$ admitting each one a basis $\B_1$ and $\B_2$, respectively, inherited from $\B$. Thus, $\mathcal{V}$ is $f$-decomposable with respect to $\B$, which is a contradiction.
\end{proof}

\begin{corollary} \label{f-inde}
Let $f : \mathcal{V} \to \mathcal{V}$ be a linear operator on a pre-euclidean space $(\mathcal{V}, \langle \cdot , \cdot \rangle)$ with basis $\mathcal{B}=\{e_i\}_{i \in I}$. Then, for each $[v_i]\in V/\sim $, the linear subspace $\mathcal{V}_{\mathcal{C}_{[v_i]}} := \bigperp_{v_j \in [v_i]}\mathbb{F}v_j$
of $\mathcal{V}$ is $f$-indecomposable  with  respect to $[v_i]$.   
\end{corollary}

\noindent We illustrate the result with this simple example.

\begin{example} \rm
Let $(\mathcal{V},\langle \cdot ,\cdot \rangle )$ be the pre-euclidean space over $\mathbb{R}$ with a fixed basis $\B = \{e_1, e_2, e_3, e_4\}$ such that $\langle e_1,e_3 \rangle =-5$, $\langle e_2,e_4 \rangle =1$, and zero on the rest. Let $f : \mathcal{V} \to \mathcal{V}$ be the linear operator defined as $$f(e_1) = f(e_3) := -2e_3, \hspace{1cm} f(e_2) = f(e_4) := 5e_2+e_4.$$
 
\noindent Then the associated graph $\Gamma(f,\B)$ is:
 
\bigskip

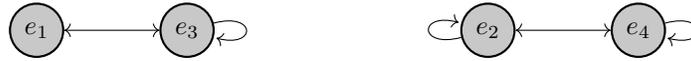
\begin{center}
\begin{tikzpicture}\label{GRAPH2}[scale=0.8, transparency group=knockout]
\begin{scope}[every node/.style={circle,thick,draw}]
\node[shape=circle,fill={rgb,255:gray,300; white,50},draw=black] (E1) at (-2,0) {$e_1$};
\node[shape=circle,fill={rgb,255:gray,300; white,50},draw=black] (E3) at (0,0) {$e_3$};
\node[shape=circle,fill={rgb,255:gray,300; white,50},draw=black] (E2) at (4,0) {$e_2$};
\node[shape=circle,fill={rgb,255:gray,300; white,50},draw=black] (E4) at (6,0) {$e_4$};

%ARROWS
\draw[<->] (E1) edge (E3);
\draw[<->] (E2) edge (E4);

%LOOPS
\path (E2) edge [loop left] (E2);
\path (E3) edge [loop right] (E3);
\path (E4) edge [loop right] (E4);
\end{scope}
\end{tikzpicture}
\end{center}

\bigskip

\noindent We have ${\mathcal V} = U_1 \perp U_2$ where $U_1,U_2$ are the subspaces with bases $\{e_1,e_3\}$ and $\{e_2,e_4\}$, respectively. Since the graph $\Gamma(f,\B)$ is not connected, we conclude $\mathcal{V}$ is $f$-decomposable  with respect to $\B$. However, $U_1$ and $U_2$ are $f$-indecomposable  with  respect to $\{e_1,e_3\}$ and $\{e_2,e_4\}$, respectively.
\end{example}

\begin{theorem} 
\label{f-indecomposable}
Let $f : \mathcal{V} \to \mathcal{V}$ be a linear operator on a pre-euclidean space $(\mathcal{V}, \langle \cdot , \cdot \rangle)$. Then for a fixed basis $\B := \{e_j\}_{j\in J}$ of ${\mathcal V}$ it holds that 
\begin{equation*}
{\mathcal V}=\bigperp _{i\in I}U_i,
\end{equation*}
being each $U_i$ a linear subspace of ${\mathcal V}$ admitting $\B_{[i]} := \{e_j : e_j \in [e_i]\}$ as a basis inherited from $\B$.
Also, we have $$f=\sum_{i\in I}f_i,$$ being each $f_i$ a linear operator on the pre-euclidean space $(\mathcal{V}, \langle \cdot , \cdot \rangle)$, for $i\in I$, such that
\begin{equation*}
f_i|_{U_i} = f|_{U_i}, \hspace{2cm} f_i(U_i) \subset U_i, \hspace{2cm} 
f_i(\perp_{i\neq j} U_j)=0.
\end{equation*}
Further, for each $i\in I$, $U_i$ is $f_i$-indecomposable with respect to $\B_{[i]}$.
\end{theorem}

\begin{proof}
Let $V/\sim =\{[e_i]\}_{i\in I}$. Then, from Equations \eqref{graphdec} and \eqref{linsub}, we can assert that $\mathcal{V} := \bigoplus_{i\in I } U_i$ is the direct sum of the family of linear subspaces $U_i  := \mathcal{V}_{\mathcal{C}_{[e_i]}}$ and $\B_{[i]} = \{e_j : e_j \in [e_i]\}$ is the basis for $U_i$, with $i\in I$. So if $U_i \neq U_j$ then $[e_i]\neq [e_j]$ and hence $\{\langle x,y\rangle ,\langle y,x \rangle \}=\{0\}$ for any $x\in [e_i]$, $y \in [e_j]$. From here, it follows that $\langle U_i,U_j \rangle = \langle U_j,U_i \rangle = 0$ for any $i\neq j$. That is, $\mathcal{V} = \bigperp_{i \in I} U_i$.

Assume now that $P_{U_j}\left(f(U_i)\right) \neq \{0\}$ for $i \neq j$, where $P_{U_j} : \mathcal{V} \to U_j$ is the projection of $\mathcal{V}$ onto $U_j$. Then there exist $e' \in [e_i]$ and $e'' \in [v_j]$ such that $P_{U_j}(f(e')) = \lambda e''+u$ with $\lambda \in \mathbb{F}\setminus \{0\}$ and $u \in \bigoplus_{e \neq e'' \in [e_j]} \mathbb{F}e \subset U_j$. Since $\lambda \neq 0$, we have $(e',e'') \in E$ and therefore $e' \sim e''$, i.e. $[e'] = [e'']$. As $[e_i] = [e']$ and $[e_j] = [e'']$, we have $[e_i] = [e_j]$ and so $U_i = U_j$, a contradiction. Thus $P_{\perp_{i\neq j}U_j}\bigl( f(U_i)\bigr) = 0$ and we conclude that $f(U_i) \subset U_i$. Consequently, each linear subspace $U_i$ of $\mathcal{V}$ admits a basis $\B_{[i]}$ inherited from $\B$.

Since $\mathcal{V}=\perp _{i\in I}U_i$, for each $i \in I$ we define the linear operator $f_i : \mathcal{V} \to \mathcal{V}$ as $f_i(U_i) := f(U_i)$ and $f_i(\bigperp_{i\neq j}U_j) := \{0\}$, so $f=\sum_{i\in I}f_i$.

Now, let us show that each $U_i$ is $f_i$-indecomposable  with respect to $\B_{[i]}$. We assume that $$U_i = U_1' \perp U_2',$$ where $U_1'$ and $U_2'$ are non-zero linear subspaces of $U_i$ admitting the $f_i$-basis $\B_1 := \{e_j : j \in J\}$ and $\B_2 := \{e_k : k \in K\}$ inherited from $\B_{[i]}$, respectively. That is, $$\B_{[i]} = \B_1 \; \dot{\cup} \; \B_2.$$ Fix some $e_j \in \B_1$ and $e_k \in \B_2$. Since $e_j$ is connected to $e_k$, there exists an undirected path $$(e_j,v_{i_2},\dots,v_{i_{n-1}},e_k)$$ from $e_j$ to $e_k$. From here, there are $e' = v_{i_s} \in \B_1$ and $e'' = v_{i_{s+1}} \in \B_2$ such that either $(e',e'') \in E$ or $(e'',e') \in E$. Since $\langle e',e'' \rangle = \langle e'',e'\rangle = 0$, we have either $P_{\mathbb{F}e''}(f_i(e')) \neq 0$ or $P_{\mathbb{F}e'}(f_i(e'')) \neq 0$. Therefore, we have either $f_i(U_1') \not\subset U_1'$ or $f_i(U_2') \not\subset U_2'$. In both cases, it is a contradiction and the proof is completed.
\end{proof}

\begin{example} \rm
Let $\mathcal{V}$ the $6$-dimensional $\mathbb{F}$-vector space with basis $\B:=\{e_1,e_2,\dots,e_6\}$ and bilinear form defined as $$\langle e_1, e_2 \rangle := \alpha, \hspace{2cm} \langle e_5, e_6 \rangle := \beta,$$ and the rest zero, with $\alpha,\beta \in \mathbb{F} \setminus \{0\}$. We consider the linear operator $f : \mathcal{V} \to \mathcal{V}$ given as $$f(e_1):=\alpha e_1, \hspace{2cm} f(e_3):= \beta e_2, \hspace{2cm} f(e_4):=-\beta e_4, \hspace{2cm} f(e_5):= -\alpha e_5+\alpha e_6,$$ and zero on the rest, with $\alpha,\beta$ the same previous scalars. By denoting $U_1,U_2,U_3$ the linear subspaces of $\mathcal{V}$ with bases $\B_{[1]} := \{e_1,e_2,e_3\}$, $\B_{[2]} :=\{e_4\}$, $\B_{[3]} := \{e_5,e_6\}$, respectively, we have $\mathcal{V}$ is $f$-decomposable   with respect to $\B$, since $${\mathcal V} = U_1 \perp U_2 \perp U_3.$$ The associated graph $\Gamma(f,\B)$ is:

\bigskip

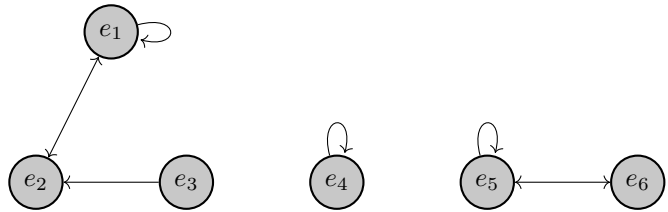
\begin{center}
\begin{tikzpicture}\label{GRAPH1C}[scale=0.8, transparency group=knockout]
\begin{scope}[every node/.style={circle,thick,draw}]
\node[shape=circle,fill={rgb,255:gray,300; white,50},draw=black] (E1) at (-0.5,2) {$e_1$};
\node[shape=circle,fill={rgb,255:gray,300; white,50},draw=black] (E2) at (-1.5,0) {$e_2$};
\node[shape=circle,fill={rgb,255:gray,300; white,50},draw=black] (E3) at (0.5,0) {$e_3$};
\node[shape=circle,fill={rgb,255:gray,300; white,50},draw=black] (E4) at (2.5,0) {$e_4$};
\node[shape=circle,fill={rgb,255:gray,300; white,50},draw=black] (E5) at (4.5,0) {$e_5$};
\node[shape=circle,fill={rgb,255:gray,300; white,50},draw=black] (E6) at (6.5,0) {$e_6$};

%ARROWS
\draw[<->] (E1) edge (E2);
\draw[->] (E3) edge (E2);
\draw[<->] (E5) edge (E6);

%LOOPS
\path (E1) edge [loop right] (E1);
\path (E4) edge [loop above] (E4);
\path (E5) edge [loop above] (E5);
\end{scope}
\end{tikzpicture}
\end{center}

\bigskip

\noindent Also, for $i \in \{1,2,3\}$, we  define $f_i : {\mathcal V} \to {\mathcal V}$ as $f_1(e_1) := \alpha e_1$, $f_1(e_3):=\beta e_2$, $f_2(e_4):=-\beta e_4$, $f_3(e_5) := -\alpha e_5 + \alpha e_6$ and zero on the rest (being  $\alpha,\beta$ the same previous non-zero scalars). We get that $$f=f_1+f_2+f_3$$ satisfies the condition of Theorem \ref{f-indecomposable}, then $U_i$ is $f_i$-indecomposable with respect to $\B_{[i]}$, for $i \in \{1,2,3\}$. So obviously, ${\mathcal V}$ is $f$-decomposable    with respect to $\B$.
\end{example}

\noindent To identify the components of the decomposition given in Theorem \ref{f-indecomposable} we only need to focus on the connected components of the associated graph.

%%%%%%%%%%%%%%%%%%%%%%%%%%%%%%%%%%%%%%%%%%%%%%%%%%%%%%%%%%%%%%%%%%
\section{Relating the graphs given by different choices of bases}\label{Sect3}
%%%%%%%%%%%%%%%%%%%%%%%%%%%%%%%%%%%%%%%%%%%%%%%%%%%%%%%%%%%%%%%%%%

In general, for a linear operator $f : \mathcal{V} \to \mathcal{V}$ on a pre-euclidean space $(\mathcal{V},\langle \cdot, \cdot \rangle )$, two different bases of $\mathcal{V}$ determine associated graphs not isomorphic, which can give rise to different decomposition of $f$ as in Theorem \ref{f-indecomposable}. We recall that two graphs $(V,E)$ and $(V',E')$ are {\em isomorphic} if there exists a bijection $\phi : V \to V'$ such that $(v_i,v_j) \in E$ if and only if $(\phi(v_i),\phi(v_j)) \in E'$. For instance, we show the next example.

\begin{example}\rm
For the linear operator and the pre-euclidean space of Example \ref{new_exam}, if we consider $w_1:=v_1+v_2$, $w_2:=v_1-v_2$, $w_3:=v_4+v_5$, $w_4:=v_4-v_5$ and $w_5:=v_3$, for the basis $\B':=\{w_1,w_2,w_3,w_4,w_5\}$ we get $$\langle w_1,w_3 \rangle = -\langle w_1,w_4 \rangle = \langle w_2,w_3 \rangle = -\langle w_2,w_4 \rangle = \langle w_3,w_5 \rangle = \langle w_4,w_5 \rangle =1,$$
$$\langle w_3,w_1 \rangle = -\langle w_3,w_2 \rangle = \langle w_4,w_1 \rangle = -\langle w_4,w_2 \rangle = \langle w_5,w_3 \rangle = -\langle w_5,w_4 \rangle = 3,$$ and also $$f(w_1) = 4w_1+2w_5, \hspace{1cm} f(w_3) = f(w_4) = \frac{1}{2} w_3 +\frac{1}{2}w_4, \hspace{1cm} f(w_5)=2w_1+w_5.$$

\bigskip

\noindent So we obtain the associated graph $\Gamma(f,\B')$ as:

\bigskip

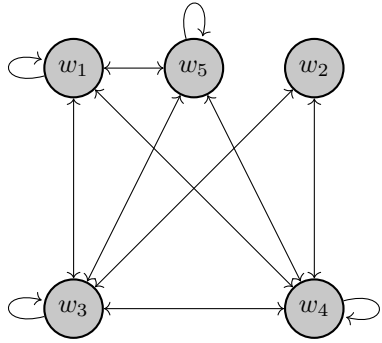
\begin{center}
\begin{tikzpicture}[scale=0.8, transparency group=knockout]
\begin{scope}[every node/.style={circle,thick,draw}]
\node[shape=circle,fill={rgb,255:gray,300; white,50},draw=black] (E+F) at (8,4) {$w_1$};
\node[shape=circle,fill={rgb,255:gray,300; white,50},draw=black] (H2) at (10,4) {$w_5$};
\node[shape=circle,fill={rgb,255:gray,300;white,50},draw=black] (E-F) at (12,4) {$w_2$};
\node[shape=circle,fill={rgb,255:gray,300;white,50},draw=black] (P+Q) at (8,0) {$w_3$};
\node[shape=circle,fill={rgb,255:gray,300;white,50},draw=black] (P-Q) at (12,0) {$w_4$};

%ARROWS
\draw[<->] (E+F) edge (H2);
\draw[<->] (P+Q) edge (P-Q);
\draw[<->] (E+F) edge (P+Q);
\draw[<->] (E+F) edge (P-Q);
\draw[<->] (E-F) edge (P+Q);
\draw[<->] (E-F) edge (P-Q);
\draw[<->] (P+Q) edge (H2);
\draw[<->] (P-Q) edge (H2);

%LOOPS
\path (E+F) edge [loop left] (E+F);
\path (H2) edge [loop above] (H2);
\path (P+Q) edge [loop left] (P+Q);
\path (P-Q) edge [loop right] (P-Q);

\end{scope}
\end{tikzpicture}
\end{center}

\bigskip

\noindent Clearly, $\Gamma(f,\B')$ is not isomorphic to the associated graph $\Gamma(f,\B)$ stated in Example \ref{new_exam}.
\end{example}

Next we give a condition under which the graphs associated to a linear operator $f : \mathcal{V} \to \mathcal{V}$ on a pre-euclidean space $(\mathcal{V},\langle \cdot, \cdot \rangle )$, performed by two different bases, are isomorphic. As a consequence, we establish a sufficient condition under which two decomposition of $f$, induced by two different bases, are equivalent.

\begin{definition}\rm
Let $f : \mathcal{V} \to \mathcal{V}$ be a linear operator on a pre-euclidean space $(\mathcal{V}, \langle \cdot , \cdot \rangle)$. Two bases $\B=\{v_i\}_{i \in I}$ and $\B'=\{w_i\}_{i \in I}$ of $\mathcal{V}$ are $f$-{\em equivalent} if there exists an automorphism $\phi : \mathcal{V} \to \mathcal{V}$ satisfying $\phi(\B)=\B'$ and $\phi \circ f = f \circ \phi$.
\end{definition}

\begin{lemma}\label{isomorphism}
Let $f : \mathcal{V} \to \mathcal{V}$ be a linear operator on a pre-euclidean space $(\mathcal{V}, \langle \cdot , \cdot \rangle)$ with basis $\mathcal{B}=\{e_i\}_{i \in I}$. Consider two bases $\B$ and $\B'$ of $\mathcal{V}$. If $\B$ and $\B'$ are $f$-equivalent bases then the associated graphs $\Gamma(f,\B)$ and $\Gamma(f,\B')$ are isomorphic.
\end{lemma}

\begin{proof}
Let us suppose that $\B=\{v_i\}_{i\in I}, \B'=\{w_i\}_{i\in I}$ are two $f$-equivalent bases of $\mathcal{V}$. Then there exists an automorphism $\phi : \mathcal{V} \to \mathcal{V}$ satisfying $f \circ \phi = \phi \circ f$ and
\begin{equation}\label{eqq1}
\langle x,y \rangle = \langle \phi(x),\phi(y) \rangle    
\end{equation}
for $x,y \in \mathcal{V}$, in such way that for every $v_i \in \B$ there exists an unique $w_{j_i} \in \B'$ verifying $\phi(v_i)= w_{j_i}$.

Let us denote by $(V,E)$ and $(V',E')$ the set of vertices and edges of $\Gamma(f,\B)$ and $\Gamma(f,\B')$, respectively. Taking into account that $V=\B$ and $V'=\B'$, and the fact $\phi(\B)=\B'$, we have that $\phi$ defines a bijection from $V$ to $V'$.
Given $v_i,v_k \in V$, we want to show that $(v_i,v_k) \in E$ if and only if $(\phi(v_i),\phi(v_k)) = (w_{j_i},w_{j_k}) \in E'$. Suppose that $(v_i,v_k)\in E$, thus either 
$\{\langle v_i,v_k \rangle , \langle v_k,v_i \rangle \} \neq \{0\}$ or $ f(v_i) = \sum_k \lambda_k v_k \mbox{ for some  } 0 \neq \lambda_k \in \mathbb{F}$.
If $\{\langle v_i,v_k \rangle , \langle v_k,v_i \rangle \} \neq \{0\}$ then by Equation \eqref{eqq1} we have
$\{\langle w_{j_i}, w_{j_k} \rangle , \langle w_{j_k},w_{j_i} \rangle \} \neq \{0\}$. If $f(v_i) = \sum_k \lambda_k v_k \mbox{ for some  } 0 \neq \lambda_k \in \mathbb{F}$, applying $\phi$ to this relation we get that $f(w_{j_i}) = f  (\phi(v_i)) = \phi (f(v_i)) = \phi (\sum_k \lambda_k v_k)=\sum_k \lambda_k \phi(v_k)=\sum_k \lambda_k w_{j_k} \mbox{ for some } 0 \neq \lambda_k \in \mathbb{F} $. Thus, $(w_{j_i},w_{j_k}) \in E'$. The same argument using $\phi^{-1}$ shows that if $(w_{j_i}, w_{j_k}) \in E'$ then $(v_i,v_k) \in E$, because also $\phi^{-1}\circ f = f \circ \phi^{-1}$. This fact concludes the proof that $\Gamma(f,\B)$ and $\Gamma(f,\B')$ are isomorphic via $\phi$.
\end{proof}

The following concept is borrowed from the theory of graded algebras (see for instance \cite{f4}).

\begin{definition}\rm
Let $(\mathcal{V},\langle \cdot ,\cdot \rangle)$ be a pre-euclidean space and let $$\Upsilon := \mathcal{V} = \bigperp_{i \in I} \mathcal{V}_i \hspace{0.4cm} \mbox{\rm and} \hspace{0.4cm} \Upsilon' := \mathcal{V} = \bigperp_{j \in J} \mathcal{V}'_j$$ be two decomposition of $\mathcal{V}$ as an orthogonal direct sum of linear subspaces. It is said that $\Upsilon$ and $\Upsilon'$ are {\em equivalent} if there exists an automorphism $\phi : \mathcal{V} \to \mathcal{V}$, and a bijection $\sigma: I \to J$ such that $\phi(\mathcal{V}_i) = \mathcal{V}'_{\sigma(i)}$ for all $i \in I$.
\end{definition}

\begin{theorem}\label{Teorema___3.1}
Let $f : \mathcal{V} \to \mathcal{V}$ be a linear operator on a pre-euclidean space $(\mathcal{V}, \langle \cdot , \cdot \rangle)$. Then for two bases $\B:=\{v_i\}_{i \in I}$ and $\B':=\{v_j'\}_{j \in J}$ of $ \mathcal{V} $ consider the following assertions:
\begin{itemize}
\item[i.] The bases $\B$ and $\B'$ are $f$-equivalent.
\item[ii.] The graphs $\Gamma(f,\B)$ and $\Gamma(f,\B')$ are isomorphic.
\item[iii.] The decomposition of the linear operator $f : \mathcal{V} \to \mathcal{V}$ with respect to $\B$      $$\Upsilon := \mathcal{V} = \bigperp_{[v_i] \in V/\sim} \mathcal{V}_{\C_{[v_i]}},$$ 
given by $$f=\sum_{i \in I}f_i$$ with $f_i|_{\mathcal{V}_{\C_{[v_i]}}} = f|_{\mathcal{V}_{\C_{[v_i]}}}$, $f_i(\mathcal{V}_{\C_{[v_i]}}) \subset \mathcal{V}_{\C_{[v_i]}}$, $f_i(\perp _{i\neq k} \mathcal{V}_{\C_{[v_k]}})=0,$ and the decomposition of $f$ with respect to $\B'$
$$\Upsilon' := \mathcal{V} = \bigperp_{[v'_j] \in V'/\sim} \mathcal{V}_{\C'_{[v'_j]}},$$ 
performed by $$f=\sum_{j \in J}f_j'$$ with $f'_j|_{\mathcal{V}_{\C'_{[v'_j]}}} = f|_{\mathcal{V}_{\C'_{[v'_j]}}}$, $f'_j(\mathcal{V}_{\C'_{[v'_j]}}) \subset \mathcal{V}_{\C'_{[v'_j]}}$, $f'_j(\perp _{i\neq k} \mathcal{V}_{\C'_{[v'_k]}})=0,$
are equivalent.
\end{itemize}
Then i. implies ii. and  iii.   
\end{theorem}

\begin{proof}
The implication from  i.\ to ii.\ was proved in Lemma \ref{isomorphism}. Let us prove the implication from  i.\ to iii. Suppose that $\phi : \mathcal{V} \to \mathcal{V}$ is an automorphism satisfying $\phi(\B)=\B'$ and $\phi \circ f = f \circ \phi$. By the implication from  i.\ to ii., we know that 
$\Gamma(f,\B)$ and $\Gamma(f,\B')$ are isomorphic via $\phi$ and thus $\phi([v])=[\phi(v)]$, for all $v\in V=\B$. It follows that $\phi(\mathcal{V}_{\C_{[v]}}) = \mathcal{V}_{\C_{[\phi(v)]}}$, for all $[v]\in V/\sim$, which proves that the decomposition of $\mathcal{V}$ corresponding to $\B$ and $\B'$ are equivalent.
\end{proof}

\begin{remark}\rm
In general, the implication ii. to i. of Theorem \ref{Teorema___3.1} (as well the converse of Lemma \ref{isomorphism}) is not valid. That is, the fact that associated graphs respect two bases are isomorphic does not imply that these two bases are $f$-equivalent. Let $\mathcal{V}$ be a pre-euclidean space with basis $\B :=\{v_1,v_2,v_3\}$ and endowed with a bilinear form defined as $$\langle v_1,v_1 \rangle = \langle v_2,v_2 \rangle = \langle v_3,v_3 \rangle = \langle v_1,v_2 \rangle :=1$$ and zero on the rest. By denoting  $w_1:=v_1+v_2$, $w_2:=v_1-v_2$ and $w_3:=v_3$ we consider the basis $\B':=\{w_1,w_2,w_3\}$, so we obtain $$\langle w_1,w_1 \rangle =3, \hspace{1cm} \langle w_2,w_2 \rangle = \langle w_3,w_3 \rangle =\langle w_2,w_1 \rangle = 1, \hspace{1cm} \langle w_1,w_2 \rangle = -1.$$ Therefore, for a zero linear operator $f$, the associated graphs $\Gamma(f,\B)$ and $\Gamma(f,\B')$ are isomorphic:

\bigskip

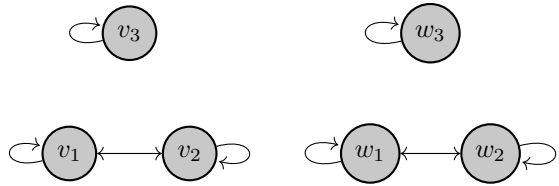
\begin{center}
\begin{tikzpicture}[scale=0.8, transparency group=knockout]
\begin{scope}[every node/.style={circle,thick,draw}]
\node[shape=circle,fill={rgb,255:gray,300; white,50},draw=black] (E3) at (-1,0) {$v_3$};
\node[shape=circle,fill={rgb,255:gray,300; white,50},draw=black] (E1) at (-2,-2) {$v_1$};
\node[shape=circle,fill={rgb,255:gray,300; white,50},draw=black] (E2) at (0,-2) {$v_2$};

\node[shape=circle,fill={rgb,255:gray,300; white,50},draw=black] (V3) at (4,0) {$w_3$};
\node[shape=circle,fill={rgb,255:gray,300; white,50},draw=black] (V1) at (3,-2) {$w_1$};
\node[shape=circle,fill={rgb,255:gray,300; white,50},draw=black] (V2) at (5,-2) {$w_2$};

%ARROWS ACIMA-ABAIXO
\draw[->] (E1) edge (E2);
\draw[->] (E2) edge (E1);
\draw[->] (V1) edge (V2);
\draw[->] (V2) edge (V1);

%LOOPS
\path (E1) edge [loop left] (E1);
\path (E3) edge [loop left] (E3);
\path (E2) edge [loop right] (E2);
\path (V1) edge [loop left] (V1);
\path (V3) edge [loop left] (V3);
\path (V2) edge [loop right] (V2);
\end{scope}
\end{tikzpicture}
\end{center}

\bigskip

\noindent However, $\B $ and $ \B'$  are not $f$-equivalent. Indeed, if there exists an isomorphism $\phi : \mathcal{V} \to \mathcal{V}$ such that $\phi(\B)=\B'$ we get, for instance, $$\phi(v_1):=w_1, \hspace{1cm} \phi(v_2):=w_3, \hspace{1cm} \phi(v_3):=w_2,$$ but $0=\langle v_1, v_3 \rangle \neq \langle \phi(v_1),\phi(v_3) \rangle = \langle w_1,w_2 \rangle = -1.$
\end{remark}

%\begin{remark} The implication iii. to i.   of Theorem \ref{Teorema___3.1} is true?  (the same for the implications iii. to ii. and  ii. to iii.) \end{remark}

%%%%%%%%%%%%%%%%%%%%%%%%%%%%%%%%%%%%%%%%%%%%%%%%%%%%%%%%%%%%%%%%%%
\section{Characterization of the minimality and weak symmetry}\label{Sect4}
%%%%%%%%%%%%%%%%%%%%%%%%%%%%%%%%%%%%%%%%%%%%%%%%%%%%%%%%%%%%%%%%%%

Let $(V,E)$ be a graph. Given $v_i,v_j \in V$, we say that a {\em directed path} from $v_i$ to $v_j$ is a sequence of vertices $(v_{i_1},\dots,v_{i_n})$ satisfying $v_{i_1} =v_i$, $v_{i_n} =v_j$ and such that $(v_{i_{r}},v_{i_{r+1}})\in E$, for $1\le r \le n-1$. We also say that $(V,E)$ is {\em symmetric} if $(v_i,v_j) \in E$ for all $(v_j,v_i)\in E$. So we present the next (weaker) concept as follows.

\begin{definition}\label{Def_weakly_symmetric}\rm
A graph $(V,E)$ is {\em weakly symmetric} if for any $(e_j,e_i)\in E$ there exists a directed path from $e_i$ to $e_j$.
\end{definition}

\noindent Of course,  every symmetric graph is weakly symmetric.

\begin{example}\rm
The following graphs are  weakly symmetric:

\bigskip

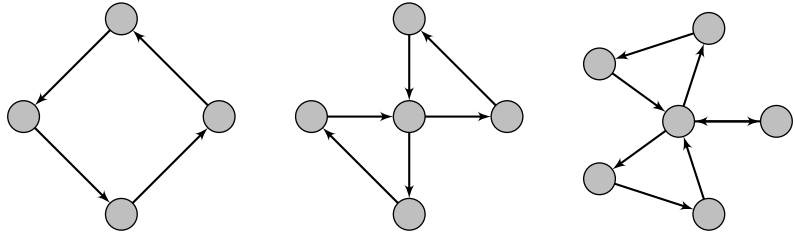
\begin{center}
\begin{tikzpicture}
	\SetGraphUnit{1.3}
	\SetVertexNormal[FillColor = gray!50,
					MinSize    = 12pt]
	\SetVertexNoLabel
	\Vertices{circle}{A,B,C,D}
	\tikzset{EdgeStyle/.style={->,> = latex'}}
	\Edges(A,B,C,D,A)
\end{tikzpicture}
\qquad
\begin{tikzpicture}
	\SetGraphUnit{1.3}
	\SetVertexNormal[FillColor = gray!50,
					MinSize    = 12pt]
	\SetVertexNoLabel
	\Vertices{circle}{A,B,D,F}
	\Vertex[x=0,y=0]{G}
	\tikzset{EdgeStyle/.style={->,> = latex'}}
	\Edges(A,B,G,F,D,G,A)
\end{tikzpicture}
\qquad
\begin{tikzpicture}
	\SetGraphUnit{1.3}
	\SetVertexNormal[FillColor = gray!50,
					MinSize    = 12pt]
	\SetVertexNoLabel
	\Vertices{circle}{A,C,D,E,F}
	\Vertex[x=0,y=0]{G}
	\tikzset{EdgeStyle/.style={->,> = latex'}}
	\Edges(A,G,C,D,G,E,F,G,A)
\end{tikzpicture}
\end{center}
\end{example}

\begin{definition}\rm
Let $f : \mathcal{V} \to \mathcal{V}$ be a linear operator on a pre-euclidean space $(\mathcal{V}, \langle \cdot , \cdot \rangle)$ with basis $\B$.
We say $\mathcal{V}$ is {\it minimal} if the unique pre-euclidean subspaces $U$ admitting a inherited basis from $\B$ such that $f(U)\subset U$ are $\{0\},\mathcal{V}$.
\end{definition}

\begin{theorem}\label{semisimplicity}
Let $f : \mathcal{V} \to \mathcal{V}$ be a linear operator on a pre-euclidean space $(\mathcal{V}, \langle \cdot , \cdot \rangle)$ with basis $\B=\{e_i\}_{i \in I}$. If $\mathcal{V}$ is minimal, then  the associated graph $\Gamma(f,\B)$ to $f$ relative to $\B$ is weakly symmetric.
\end{theorem}

\begin{proof}
If $\mathcal{V}$ is minimal we have that ${\mathcal V}$ is the unique non-zero pre-euclidean subspace that satisfies $f(\mathcal{V}) \subset \mathcal{V}$.
Let $\Gamma(f,\B):=(V,E)$ be the associated graph and take some $(e_i,e_j)\in E$. Therefore, either $\langle e_i,e_j \rangle \neq  0 $ or $ \langle e_j,e_i\rangle \neq  0 $ or $f(e_i) = \sum_j \lambda_j e_j$ for some $0 \neq \lambda_j \in \mathbb{F}$. The first two cases imply $(e_j,e_i) \in E$. In the last case, we have $f \neq 0$. Let us now define $$\B_j :=\{e_k \in \B : e_k = e_j \mbox{ or there exists a directed path from } e_j \mbox{ to } e_k\}.$$ Since $e_j \in \B_j$, we have $\B_j \neq \emptyset$. So let $U$ be the space spanned by $\B_j$. Let $e_t \in \B_j$ and $f(e_t) = \sum_m \lambda_m e_m$. If $\lambda_m \neq 0$ we get $(e_t,e_m) \in E$ and hence there exists a directed path from $e_j$ to $e_m$:

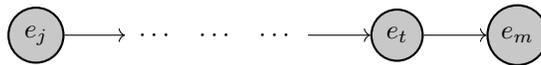
\bigskip

\begin{center}
\begin{tikzpicture}[scale=0.8, transparency group=knockout]
\begin{scope}[every node/.style={circle,thick,draw}]
\node[shape=circle,fill={rgb,255:gray,300; white,50},draw=black] (EJ) at (0,0) {$e_j$};
\node[shape=circle,fill=white,draw=white] (BLANC) at (2,0) {$\cdots$};
\node[shape=circle,fill=white,draw=white] (BLANC3) at (3,0) {$\cdots$};
\node[shape=circle,fill=white,draw=white] (BLANC2) at (4,0) {$\cdots$};
\node[shape=circle,fill={rgb,255:gray,300; white,50},draw=black] (ET) at (6,0) {$e_t$};
\node[shape=circle,fill={rgb,255:gray,300; white,50},draw=black] (EM) at (8,0) {$e_m$};

%ARROWS ACIMA-ABAIXO
\draw[->] (EJ) edge (BLANC);
%\draw[-] (BLANC)   edge[dotted] (BLANC2);
\draw[->] (BLANC2) edge (ET);
\draw[->] (ET) edge (EM);
\end{scope}
\end{tikzpicture}
\end{center}

\bigskip
\noindent So it implies $e_m\in \B_j$ and therefore $f(U)\subset U$. Since $\B_j \neq \emptyset$ and ${\mathcal V}$ is minimal we have $U=\mathcal{V}$, so we conclude  $\B_j = \B$ and then $e_i \in \B_j$. Thus there exists a directed path from $e_j$ to $e_i$ as required.
\end{proof}
 
\begin{corollary}
Let $f : \mathcal{V} \to \mathcal{V}$ be a linear operator on a pre-euclidean space $(\mathcal{V}, \langle \cdot , \cdot \rangle)$ with basis $\B=\{e_i\}_{i \in I}$. If $\mathcal{V}$ is minimal, then the associated graph $\Gamma(f,\B)$ to $f$ relative to $\B$ is connected.    
\end{corollary}
\begin{proof}
    This is an immediate consequence of Corollary \ref{f-inde}, Theorem \ref{f-ind} and Theorem \ref{f-indecomposable}.
\end{proof}
\begin{remark}\rm
In general, the converse of the previous results are not valid. As counterexample, let $(\mathcal{V},\langle \cdot ,\cdot \rangle )$ be the pre-euclidean space over $\mathbb{R}$ with basis $\B := \{e_1, e_2, e_3, e_4\}$ and bilinear form given as $$\langle e_1,e_2 \rangle =\langle e_2,e_3 \rangle =\langle e_3,e_4 \rangle = \langle e_4,e_1 \rangle := 1.$$ Let $f : \mathcal{V} \to \mathcal{V}$ be the linear operator defined as $f(e_i) := e_i$ for $i \in \{1,2,3,4\}$. Then $\mathcal{V}$ is not minimal since, for instance, the pre-euclidean subspace $U$ with inherited basis $\{e_1\}$ from $\B$ satisfies $f(U)\subset U$. However, the associated graph $\Gamma(f,\B)$ is  clearly connected and weakly symmetric:

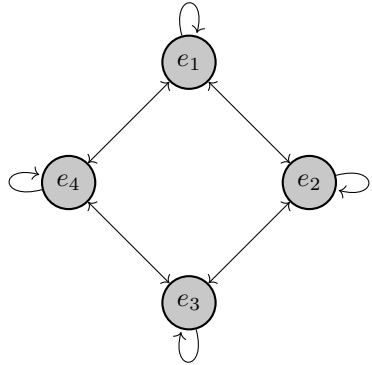
\begin{center}
\begin{tikzpicture}[scale=0.8, transparency group=knockout]
\begin{scope}[every node/.style={circle,thick,draw}]
\node[shape=circle,fill={rgb,255:gray,300; white,50},draw=black] (V1) at (0,4) {$e_1$};
\node[shape=circle,fill={rgb,255:gray,300; white,50},draw=black] (V2) at (2,2) {$e_2$};
\node[shape=circle,fill={rgb,255:gray,300; white,50},draw=black] (V3) at (0,0) {$e_3$};
\node[shape=circle,fill={rgb,255:gray,300; white,50},draw=black] (V4) at (-2,2) {$e_4$};

%ARROWS ACIMA-ABAIXO
\draw[<->] (V1) edge (V2);
\draw[<->] (V2) edge (V3);
\draw[<->] (V3) edge (V4);
\draw[<->] (V4) edge (V1);

%LOOPS
\path (V1) edge [loop above] (V1);
\path (V2) edge [loop right] (V2);
\path (V3) edge [loop below] (V3);
\path (V4) edge [loop left] (V4);
\end{scope}
\end{tikzpicture}
\end{center}
\end{remark}

%%%%%%%%%%%%%%%%%%%%%%%%%%%%%%%%%%%%%%%%%%%%%%%%%%%%%%%%%%%%%%%%%%%%%%%%%%%%%%%%%%%%%%%%%%%%%
\section{Conclusion}
\label{sec0}
%%%%%%%%%%%%%%%%%%%%%%%%%%%%%%%%%%%%%%%%%%%%%%%%%%%%%%%%%%%%%%%%%%%%%%%%%%%%%%%%%%%%%%%%%%%%%
We were motivated by the following situation: If we consider a family  of linear operators $$\{f_i : \mathcal{V}_i \to \mathcal{V}_i\}_{i\in I},$$ where any $V_i$ is a pre-euclidean space, then we can construct in a natural way a new linear operator $$f: \bigoplus_{i \in I} \mathcal{V}_i \to \bigoplus_{i \in I} \mathcal{V}_i$$ as $f=\sum_{i\in I} f_i$,   where $\mathcal{V} := \bigoplus_{i\in I} \mathcal{V}_i$ is the pre-euclidean space defined by  componentwise operations.

In this paper, our purpose has been to study the converse problem, and we have given a positive answer by proving in Theorem \ref{f-indecomposable} that given a linear operator $f : \mathcal{V} \to \mathcal{V}$ on a pre-euclidean space, it is possible to find a family of pre-euclidean spaces $\{\mathcal{V}_i\}_{i\in I}$, and a family of linear operators $\{f_i : \mathcal{V}_i \to \mathcal{V}_i\}_{i\in I} $ in such a way that $\mathcal{V} = \perp_{i\in I} \mathcal{V}_i$ and $f = \sum_{i\in I} f_i$.
In order to approach our question, we have used techniques of graphs. This allows us to obtain the above decomposition of the pre-euclidean space $\mathcal{V}$ and the linear operator $f$ in a very practical way, just by looking to the graph $\Gamma(F,\B)$ associated to $f$ (and a fixed basis $\B$). Also, the minimality of our structure is characterized in Theorem \ref{semisimplicity}, which says that, 
for a linear operator $f : \mathcal{V} \to \mathcal{V}$ on a minimal pre-euclidean space $(\mathcal{V}, \langle \cdot , \cdot \rangle)$ (with basis $\B=\{e_i\}_{i \in I}$), then  the associated graph $\Gamma(f,\B)$ to $f$ relative to $\B$ is weakly symmetric.

Finally, we would like to note that  future research in this topic are to generalize this result for different classes of operators (not necessarily linear operators) and, of course, consider operators on structures different to pre-euclidean spaces  (for instance Banach spaces).\\

{\bf Acknowledgment.} The authors would like to thank the reviewers
for the comments and suggestions that helped to improve the work.


\begin{thebibliography}{99}

\bibitem{Leibniz_grafos} Barreiro, E; Calder\'on, A.J.; Lopes, S.A.; S\'anchez, J.M. Leibniz algebras and graphs, {\em Linear Multilinear Algebra}. DOI: 10.1080/03081087.2022.2092048

\bibitem{Yografos} Calder\'on, A.J.; Navarro, F.J. Bilinear maps and graphs. {\em Discrete Appl. Math.} {\bf 2019}, 263, 69--78.

\bibitem{apply6} Morales, J.V.S.; Palma, T.M. On quantum adjacency algebras of Doob graphs and their irreducible modules. {\em J. Algebraic Combin.} {\bf 2021}, 54, 979–998.

\bibitem{apply8} Qaralleh, I.; Mukhamedov, F. Volterra evolution algebras and their graphs. {\em Linear Multilinear Algebra} {\bf 2021}, 69, 2228–2244.

\bibitem{apply13} Zhilina, S. Orthogonality graphs of real Cayley-Dickson algebras. Part I: Doubly alternative zero divisors and their hexagons. {\em Internat. J. Algebra Comput.} {\bf 2021}, 31, 663–689.

\bibitem{graph_oper1} Ambrozie, C.; Bra\v{c}i\v{c}, J.; Kuzma, B.; M\"{u}ller, V. The commuting graph of bounded linear operators on a Hilbert space. {\em J. Funct. Anal.} {\bf 2013}, 264, 1068--1087.

\bibitem{graph_oper2} Bailey, S.; Beasley, L.B. Linear operators on graphs which preserve the dot-product dimension. {\em J. Combin. Math. Combin. Comput.} {\bf 2016}, 98, 31--42.

\bibitem{graph_oper3} Beasley, L.B. Linear operators on graphs: genus preservers. {\em Congr. Numer.} {\bf 2018}, 231, 5--13.

\bibitem{graph_oper4} Beasley, L.B.; Brown, D.E. Linear operators on graphs that preserve the clique cover number. {\em Congr. Numer.} {\bf 2015}, 225, 83--94.

\bibitem{graph_oper5} Beasley, L.B.; Pullman, N.J. Linear operators preserving properties of graphs. {\em Proceedings of the Twentieth Southeastern Conference on Combinatorics, Graph Theory, and Computing} (Boca Raton, FL, 1989). {\em Congr. Numer.} {\bf 1990}, 70, 105--112.

\bibitem{graph_oper6} Garber, A.I. Graphs of linear operators. (Russian) {\em Tr. Mat. Inst. Steklova} {\bf 263} (2008), Geometriya, Topologiya i Matematicheskaya Fizika. I, 64–71 ISBN: 5-7846-0108-3; 978-5-7846-0108-7; translation in {\em Proc. Steklov Inst. Math.} {\bf 2008}, 263, 57--64.

\bibitem{graph_oper7} Zeng, S.B.; Cao, J.L. Closed graph properties of set-valued linear operators. (Chinese) {\em J. Tianjin Univ.} {\bf 1994}, 27, 799--803.

\bibitem{apply12} Yang, F.; Sun, Q.; Zhang, C. Semisymmetric Graphs Defined by Finite-Dimensional Generalized Kac–Moody Algebras. {\em Bull. Malays. Math. Sci. Soc.} {\bf 2022}, 45, 3293–3305.

\bibitem{apply1} Belishev, M.I.; Kaplun, A.V. Canonical representation of the $C^*$-algebra of eikonals related to the metric graph. (Russian) {\em Izv. Ross. Akad. Nauk Ser. Mat.} {\bf 2022}, 86, 3--50.

\bibitem{apply2} Chhiti, M.; Kaiba, K. The total graph of amalgamated algebras. {\em Indian J. Math.} {\bf 2022}, 64, 245–261.

\bibitem{apply3} Christoffersen, N.J.; Dutkay, D.E. Representations of Cuntz algebras associated to random walks on graphs. {\em J. Operator Theory} {\bf 2022}, 88, 139--170.

\bibitem{apply5} Matsumoto, K. On simplicity of the $C^*$-algebras associated with $\lambda$-graph systems. {\em J. Math. Anal. Appl.} {\bf 2022}, 515, Paper No. 126441, 20 pp.

\bibitem{apply7} Myers, T. Constructing Clifford algebras for windmill and dutch windmill graphs; a new proof of the friendship theorem. {\em Combinatorics, graph theory and computing}, 47--81, {\em Springer Proc. Math. Stat.}, 388, Springer, Cham, 2022.

\bibitem{apply9} Ramos, E. The graph minor theorem meets algebra. {\em Notices Amer. Math. Soc.} {\bf 2022}, 69, 1297--1305.

\bibitem{apply10} Rastgar, Z.; Khashyarmanesh, K.; Afkhami, M. On the comaximal (ideal) graph associated with amalgamated algebra. {\em Discrete Math. Lett.} {\bf 2023}, 11, 38--45.

\bibitem{apply11} Reis, T.; Cadavid, P. Derivations of evolution algebras associated to graphs over a field of any characteristic. {\em Linear Multilinear Algebra} {\bf 2022}, 70, 2884--2897.

\bibitem{apply4} Larson, D.M. Connectedness of graphs arising from the dual Steenrod algebra. {\em J. Homotopy Relat. Struct.} {\bf 2022}, 17, 145--161.

\bibitem{f4} Elduque, A.;  Kochetov, M.  Gradings on simple Lie algebras. Mathematical Surveys and Monographs, 189. American Mathematical Society, Providence, RI; Atlantic Association for Research in the Mathematical Sciences (AARMS), Halifax, NS, 2013. xiv+336 pp.

\end{thebibliography}
\end{document}